\documentclass[preprint,12pt]{elsarticle}
\usepackage{amsthm,amsmath,amssymb}
\usepackage[colorlinks=true,citecolor=black,linkcolor=black,urlcolor=blue]{hyperref}
\usepackage{graphicx}
\usepackage{float}
\usepackage{epstopdf}
\usepackage{epsfig}
\usepackage{subfigure}

\theoremstyle{plain}
\newtheorem{theorem}{Theorem}
\newtheorem{lemma}[theorem]{Lemma}

\theoremstyle{definition}
\newtheorem{definition}[theorem]{Definition}

\theoremstyle{remark}
\newtheorem{remark}[theorem]{Remark}

\title{Eulerian and bipartite binary delta-matroids}
\author{Qi Yan~~~Xian'an Jin\footnote{Corresponding author.}\\
\small School of Mathematical Sciences\\[-0.8ex]
\small Xiamen University\\[-0.8ex]
\small P. R. China\\
\small\tt Email:qiyanmath@163.com;xajin@xmu.edu.cn}
\date{}

\begin{document}

\begin{abstract}
Delta-matroid theory is often thought of as a generalization of topological graph theory. It is well-known that an orientable embedded graph is bipartite if and only if its Petrie dual is orientable. In this paper, we first introduce the concepts of Eulerian and bipartite delta-matroids and then extend the result from embedded graphs to arbitrary binary delta-matroids. The dual of any bipartite embedded graph is Eulerian. We also extend the result from embedded graphs to the class of delta-matroids that arise as twists of binary matroids. Several related results are also obtained.
\end{abstract}
\begin{keyword}
matroid\sep delta-matroid\sep binary\sep bipartite\sep Eulerian
\vskip0.2cm
\MSC 05B35\sep 05C10\sep 57M15
\end{keyword}
\maketitle

\section{Introduction}
Matroid theory is often thought of as a generalization of graph theory. Many results in graph theory turn out to be special cases of results in matroid theory.
A graph is said to be \emph{Eulerian} if the degree of each of its vertices is even, and a graph is \emph{bipartite} if it does not contain cycles of odd lengths. Welsh \cite{DW} introduced the concepts of Eulerian and bipartite on matroids and showed how Euler's well-known graph theorem has an appealing matroid generalization. They showed that for binary matroids the properties of being Euler and bipartite are dual concepts, thus generalizing Euler's theorem for graphs. Wilde \cite{PJW} showed that if $M$ is a binary matroid, then every cocircuit of $M$ has even cardinality if and only if $M$ can be obtained by contracting some other binary matroid $M^{+}$ onto a single circuit. This is the natural analog of the Euler circuit theorem for graphs. More recently, Shikare and Raghunathan \cite{SR} have shown that a binary matroid $M$ is Eulerian if and only if the number of independent sets of $M$ is odd. Shikare \cite{MMS} also gave some new characterizations of Eulerian and bipartite binary matroids.

However, our interest here is in a generalisation of matroids called delta-matroids. Delta-matroids were introduced by Bouchet in \cite{AB1}. Chun, Moffatt, Noble and Rueckriemen \cite{CISR, CMNR} gave an analogous correspondence between embedded graphs and delta-matroids. Many fundamental definitions and results in topological graph theory and delta-matroid theory are compatible with each other. A significant consequence of this connection is that the geometric ideas of topological graph theory provide insight and intuition into the structure of delta-matroids, thus pushing forward the development of both areas.

The key purpose of this paper is to propose and study a similar correspondence between  bipartite (Eulerian) embedded graphs and delta-matroids.
We know an embedded graph $G$ is bipartite (Eulerian) if and only if the underlying graph of $G$ is bipartite (Eulerian). Hence, we call that a delta-matroid $D=(E, \mathcal{F})$ is  bipartite (Eulerian), if the lower matroid of $D$ is  bipartite (Eulerian). When applied to matroids, this definition of bipartite (Eulerian) delta-matroids agrees with the usual definition of bipartite (Eulerian) matroids.
Let $M=(E, \mathcal{F})$ be a binary matroid. Brijder and Hoogeboom \cite{RH2} showed that $M$ is bipartite if and only if  the loop complementation of $M$ on $E$ is an even delta-matroid. The result is interesting in the context of embedded graphs. In \cite{EM1} it was shown that an orientable embedded graph $G$ is bipartite if and only if its Petrie dual is orientable. We extend the nomenclature from embedded graphs to arbitrary binary delta-matroids.  It is well-known fact that for a plane graph $G$, $G$ is bipartite if and only if the dual of $G$ is Eulerian. This fact only holds for general embedded graphs in one direction: the dual of any bipartite embedded graph is Eulerian. And it is not true in other direction in general. We also extend the result from embedded graphs to the class of delta-matroids that arise as twists of binary matroids.

\section{Preliminaries}

We begin by defining matroids, delta-matroids and associated  terminologies. All definitions follow \cite{AB3, CISR, CMNR, JO}.
Throughout this paper, we will often omit the set brackets in the case of a single element set. For example, we write $F\cup e$ instead of $F\cup \{e\}$,
or $E-e$ instead of $E-\{e\}$.

\subsection{Matroids}

A \emph{matroid} $M$ is an ordered pair $(E, \mathcal{I})$ consisting of a finite set $E$ and a collection $\mathcal{I}$ of subsets of $E$ having the following three properties:

(I1.) $\emptyset \in \mathcal{I}$.

(I2.) If $I_{1}\in \mathcal{I}$, and $I_{2}\subseteq I_{1}$, then $I_{2}\in \mathcal{I}$.

(I3.) If $I_{1}$ and $I_{2}$ are in $\mathcal{I}$, and $|I_{2}|<|I_{1}|$, then there is an element $e\in I_{1}-I_{2}$ such that $I_{2}\cup e\in \mathcal{I}$.

If $M$ is the matroid $(E, \mathcal{I})$, then $M$ is called a matroid on $E$. The members of $\mathcal{I}$ are the \emph{independent sets} of $M$, and $E$ is the \emph{ground set} of $M$. We shall often write $\mathcal{I}(M)$ for $\mathcal{I}$ and $E(M)$ for $E$, particularly when several matroids are being considered.

 A maximal independent set of $M$ is a \emph{base} of $M$. A subset of $E$ not belonging to $\mathcal{I}$ is said to be \emph{dependent}. A minimal dependent subset of $E$ is called a \emph{circuit} of $M$. The set of all bases and circuits of $M$ will be denoted by $\mathcal{B}(M)$ and $\mathcal{C}(M)$, respectively. We say that the \emph{rank} of $M$, written $r(M)$, is equal to $|B|$ for any $B\in\mathcal{B}(M)$.

 The matroid $M^{*}$, whose ground set is $E(M)$ and whose set of bases is $\mathcal{B}(M^{*}):=\{E-B: B\in \mathcal{B}\}$, is called the \emph{dual} of $M$. A circuit of $M^{*}$ is called a \emph{cocircuit} of $M$.

\begin{definition}\cite{DW}
 A matroid $M=(E, \mathcal{I})$ is said to be \emph{Eulerian} if there are disjoint circuits $C_{1}, \cdots, C_{p}$ in $M$ such that $E(M)=C_{1}\cup \cdots \cup C_{p}$. A matroid is said to be \emph{bipartite} if its every circuit has even cardinality.
\end{definition}

 \subsection{Set systems and delta-matroids}

A \emph{set system} is a pair $D=(E, \mathcal{F})$, where $E$ is a finite set, which we call the \emph{ground set}, and $\mathcal{F}$ is a collection of  subsets of $E$, called \emph{feasible sets}. We define $E(D)$ to be $E$ and $\mathcal{F}(D)$ to be $\mathcal{F}$. A set system $(E, \mathcal{F})$ is \emph{proper} if
$\mathcal{F}$ is not empty. It is \emph{trivial} if $E$ is empty.

For sets $X$ and $Y$, their \emph{symmetric difference}is denoted by $X\Delta Y$ and is defined to be $(X\cup Y)-(X\cap Y)$.

A \emph{delta-matroid} is a proper set system $D=(E, \mathcal{F})$ that satisfies the symmetric exchange axiom:

{\bf Axiom} (Symmetric exchange axiom). For all $(X, Y, u)$ with $X, Y \in \mathcal{F}$ and $u\in X\Delta Y$, there is an element $v\in X\Delta Y$ such that
$X\Delta \{u, v\}$ is in $\mathcal{F}$.

Note that we allow $u=v$ in symmetric exchange axiom. These structures were first studied by Bouchet in \cite{AB1}.
If all of the feasible sets of a delta-matroid are equicardinal, then the delta-matroid is a matroid and we refer to its feasible sets as its bases.

For a delta-matroid $D=(E, \mathcal{F})$, let $\mathcal{F}_{max}(D)$ and $\mathcal{F}_{min}(D)$ be the set of feasible sets with maximum and minimum cardinality, respectively. We will usually omit $D$ when the context is clear. Let $D_{max}=(E, \mathcal{F}_{max})$ and let $D_{min}=(E, \mathcal{F}_{min})$.
Then $D_{max}$ is the \emph{upper matroid} and $D_{min}$ is the \emph{lower matroid} for $D$. These were defined by Bouchet in \cite{AB2}. It is straightforward to show that the upper matroid and the lower matroid are indeed matroids.
\begin{definition}
A delta-matroid $D=(E, \mathcal{F})$ is  bipartite (Eulerian), if $D_{min}$ is  bipartite (Eulerian).
\end{definition}

For delta-matroids (or matroids) $D_{1}=(E_{1}, \mathcal{F}_{1})$ and $D_{2}=(E_{2}, \mathcal{F}_{2})$, where $E_{1}$ is disjoint from $E_{2}$, the \emph{direct sum} of $D_{1}$ and $D_{2}$, written $D_{1}\oplus D_{2}$, is constructed by
$$D_{1}\oplus D_{2}:=(E_{1}\cup E_{2}, \{F_{1}\cup F_{2}: F_{1}\in \mathcal{F}_{1}, F_{2}\in \mathcal{F}_{2}\}).$$
In particular, if $D_{1}$ and $D_{2}$ are matroids, then $$\mathcal{C}(D_{1}\oplus D_{2})=\mathcal{C}(D_{1})\cup \mathcal{C}(D_{2}).$$
\subsection{Minors}
For a proper set system $D=(E, \mathcal{F})$, and $e\in E$, if $e$ is in every feasible set of $D$, then we say that $e$ is a \emph{coloop} of $D$. If $e$ is in no feasible set of $D$, then we say that $e$ is a \emph{loop} of $D$.

If $e$ is not a coloop, then we define $D$ \emph{delete} $e$, written $D\backslash e$, to be
$$(E-e, \{F: F\in \mathcal{F}, F\subseteq E-e\}).$$ If $e$ is not a loop, then we define $D$ \emph{contract} $e$, written $D/e$, to be $$(E-e, \{F-e: F\in \mathcal{F}, e\in F\}).$$ If $e$ is loop or a coloop, then one of $D\backslash e$ and $D/e$ has already been defined, so we can set $D/e=D\backslash e$.

If $D$ is a delta-matroid then both $D\backslash e$ and $D/e$ are delta-matroids (see \cite{AB3}). Let $D'$ be a delta-matroid obtained from $D$ by a sequence of deletions and contractions. Then $D'$ is independent of the order of the deletions and contractions used in its construction (see \cite{AB3}) and $D'$ is called a \emph{minor} of $D$. If $D'$ is a minor of $D$ formed by deleting the elements of $X$ and contracting the elements of $Y$ then we write $D'=D\backslash X/Y$. If the sizes of the feasible sets of a delta-matroid all have the same parity, then we say that the delta-matroid is \emph{even}. Otherwise, we say that the delta-matroid is \emph{odd}. Note that if $D$ is even, then so are its minors. The \emph{restriction} of $D$ to a subset $A$ of $E$, written $D|_{A}$, is equal to $D\setminus (E-A)$.

\subsection{Twists and loop complementations of set systems}

Twists are one of the fundamental operations of delta-matroid theory. Let $D=(E, \mathcal{F})$ be a set system. For $A\subseteq E$, the \emph{twist} of $D$ with respect to $A$, denoted by $D*A$, is given by $(E, \{A\Delta X: X\in \mathcal{F}\})$. The \emph{dual} of $D$, written $D^{*}$, is equal to $D*E$. It follows easily from the identity $(F'_{1}\Delta A)\Delta (F'_{2}\Delta A)=F'_{1}\Delta F'_{2}$ that the twist of a delta-matroid is also a delta-matroid, as Bouchet showed in
\cite{AB1}. However, if $D$ is a matroid, then $D*A$ need not be a matroid. Note that if $e\in E$, then $D/e=(D*e)\backslash e$ and $D\backslash e=(D*e)/e$. Moreover, if $X\subseteq E$, then $D\setminus X=(D^{*}/X)^{*}$ \cite{CMNR}.

Following Brijder and Hoogeboom \cite{RH}, let $D=(E, \mathcal{F})$ be a set system and $e\in E$. Then $D+e$ is defined to be the set system $(E, \mathcal{F}')$, where $$\mathcal{F}':=\mathcal{F}\Delta \{F\cup e: F\in \mathcal{F}, e\notin F\}.$$ If $e_{1}, e_{2}\in E$ then $(D+e_{1})+e_{2}=(D+e_{2})+e_{1}$. This means that if $A=\{e_{1}, \cdots, e_{n}\}\subseteq E$ we can unambiguously define the \emph{loop complementation} of $D$ on $A$, by $D+A:=D+e_{1}+\cdots +e_{n}$.
Let $D=(E, \mathcal{F})$ be a set system and $X, Y \subseteq E$. Brijder and Hoogeboom  \cite{RH} showed that $Y\in \mathcal{F}(D+X)$ if and only if $|\{Z \in \mathcal{F} : Y\backslash X\subseteq Z \subseteq Y\}|$ is odd.

\subsection{Representability}

Let $A=(A_{vw}: v, w\in E)$ be a square matrix with coefficients in a field $Q$. We say that $A$ is \emph{antisymmetric} if $A_{vw}=-A_{wv}$ for every $v, w\in E$ and $A_{vv}=0$ for every $v\in E$. We say that $A$ is \emph{quasisymmetric} if there exists a function $\varepsilon: E\rightarrow \{-1, +1\}$ such that $\varepsilon(v)A_{vw}=\varepsilon(w)A_{wv}$ holds for every $v, w\in E$. In particular, if $\varepsilon$ is constant function, $A$ is a symmetric matrix. If $A$ is either an antisymmetric or quasisymmetric matrix, it is called a matrix of symmetric type. For every $X\subseteq E$ we let $A[X]=(A_{vw}: v, w\in X)$. By convention we consider $A[\emptyset]$ as a nonsingular matrix. Let $D(A)=(E, \{X: X\subseteq E, A[X]~\mbox{is nonsingular}\})$. Bouchet showed in \cite{AB4} that $D(A)$ is indeed a delta-matroid. A strong representation of the delta-matroid $D$ is a matrix $A$ of symmetric type over a field $Q$ such that $D=D(A)$. A necessary condition for $D$ to have a strong representation is that $\emptyset$ is feasible, and then we say that $D$ is a \emph{normal} delta-matroid. Bouchet showed in \cite{AB3} that if two normal delta-matroids are twist and one of them has a strong representation $A$ over a field $Q$, then the other one has a strong representation $A'$ over $Q$. Moreover $A$ and $A'$ are either both antisymmetric or both quasisymmetric. The delta-matroid $D$ is said to be representable over the field $Q$ if there exists a twist normal delta-matroid $D'$ which has a strong representation over $Q$ . If $A$ is an antisymmetric matrix, then any feasible set of $D(A)$ has an even cardinality (recall that a nonsingular antisymmetric matrix has an even order). A delta-matroid is said to be \emph{binary} if it has a twist that is isomorphic to $D(A)$ for some symmetric matrix $A$ over $GF(2)$. Binary delta-matroids from an important class of delta-matroids. In \cite{AB3} it was shown that if $D$ is a binary delta-matroid, then both of $D_{min}$ and $D_{max}$ are binary matroids.
Moreover, the minor of a binary delta-matroid is also a binary delta-matroid.

As it is much more convenient for our purposes, we realize cellularly embedded graphs as ribbon graphs. We give a brief review of ribbon graphs referring the reader to \cite{EM1,EM} for further details.

\begin{definition}[\cite{EM}]
A \emph{ribbon graph} $G=(V(G), E(G))$ is a (possibly non-orientable) surface with boundary
represented as the union of two sets of discs, a set $V(G)$ of vertices, and a set $E(G)$ of edges
such that
\begin{enumerate}
\item The vertices and edges intersect in disjoint line segments;
\item Each such line segment lies on the boundary of precisely one vertex and precisely one edge;
\item Every edge contains exactly two such line segments.
\end{enumerate}
\end{definition}

An edge $e$ of a ribbon graph is a \emph{loop} if it is incident with exactly one vertex. A loop (respectively, cycle) is \emph{non-orientable} if together with its incident vertex (vertices) it forms a M\"{o}bius band, and is \emph{orientable} otherwise.
A ribbon graph is \emph{non-orientable} if it contains a non-orientable loop or cycle, and is \emph{orientable} otherwise.

\section{Main results}
Let $M=(E, \mathcal{F})$ be a binary matroid. Brijder and Hoogeboom \cite{RH2} showed that $M$ is bipartite if and only if $M+E$ is an even delta-matroid.
The result is interesting in the context of ribbon graphs. In \cite{EM1} it was shown that an orientable ribbon graph $G$ is bipartite if and only if its Petrie dual is orientable. Chun et al. \cite{CISR,CMNR} showed that partial duals and twists as well as partial Petrials and
loop complementations are compatible. So it should come as no surprise that twisted duality (see \cite{EM1})
for ribbon graphs and for delta-matroids are compatible as well.
Firstly, we give some basic lemmas and extend the nomenclature from ribbon graphs to arbitrary binary delta-matroids.

\begin{lemma}[\cite{CMNR}]\label{lea}
Let $D=(E, \mathcal{F})$ be a delta-matroid, let $A\subseteq E$ and let $s_{0}=\min \{|B\cap A|:B\in \mathcal{B}(D_{min})\}$. Then for any $F\in \mathcal{F}$ we have $|F\cap A|\geq s_{0}$.
\end{lemma}

\begin{lemma}\label{leb}
Let $D=(E, \mathcal{F})$ be a  set system and $e\in E$. Then $$(D\setminus e)_{min}=D_{min}\setminus e.$$
\end{lemma}

\begin{proof}
If $e$ is not a coloop of  $D_{min}$, then
$$\mathcal{F}((D\setminus e)_{min})=\{F: F\in \mathcal{F}(D_{min}), e\notin F\}=\mathcal{F}(D_{min}\setminus e).$$
Otherwise, $e$ is a coloop of $D_{min}$. Obviously, $e$ is also a coloop of $D$ by Lemma \ref{lea}. Then
$$\mathcal{F}((D\setminus e)_{min})=\{F-e: F\in \mathcal{F}(D_{min})\}=\mathcal{F}(D_{min}\setminus e).$$
\end{proof}

\begin{remark}
$(D/e)_{min}\neq D_{min}/e$. For example, let $$D=(\{1, 2\}, \{\emptyset, \{1, 2\}\}).$$
Then $(D/1)_{min}=(\{2\}, \{\{2\}\})$ but $D_{min}/1=(\{2\}, \{\emptyset\}).$
\end{remark}

In the context of ribbon graphs, a ribbon graph $G$ is non-orientable if and only if there exists a non-orientable cycle (or loop) of $G$. We extend the result from ribbon graphs to delta-matroids as shown in following lemma.

\begin{lemma}\label{lec}
Let $D=(E, \mathcal{F})$ be a binary delta-matroid. Then $D$ is an odd delta-martoid if and only if there exists a circuit $C$ of $D_{min}$ such that $C\in \mathcal{F}(D|_{C})$.
\end{lemma}

\begin{proof}
($\Rightarrow$) Let $C$ be a subset of $E$ such that $D|_{C}$ is an odd delta-matroid and  $D|_{C-e}$ is an even delta-matroid for any $e\in C$.
There exists a feasible set $F$ of $D|_{C}$ such that the size of $F$ and the rank of $(D|_{C})_{min}$ have different parity. For any $e\in C$, it obvious that $e$ is not a coloop of $D|_{C}$. Otherwise  $D|_{C-e}$ is also an odd delta-matroid, contradicting the choice of $C$. Moreover, $e$ is not a coloop of $(D|_{C})_{min}$ by Lemma \ref{lea}. Hence, there is a base $B$ of  $(D|_{C})_{min}$ such that $e\notin B$. If $e\notin F$, then $B, F\in \mathcal{F}(D|_{C-e})$. It follows that $D|_{C-e}$ is an odd delta-matroid by $|B|$ and $|F|$ having opposite parity. We can see that $e\in F$, then $F=C$.
Thus $C\in \mathcal{F}(D|_{C})$.

If $C$ is a single element set, we may assume that $C=\{e\}$, then $$D|_{C}=(\{e\}, \{\emptyset, \{e\}\}).$$ The necessity is easily verified.
Otherwise, $|C|\geq 2$. Suppose that  $C$ is not a circuit of $D_{min}$. Note that $(D|_{C})_{min}=D_{min}|_{C}$ by Lemma \ref{leb}. Thus, $C$ is not a circuit of $(D|_{C})_{min}$.  Since $C$ is not a circuit of $(D|_{C})_{min}$ and any element of $C$ is not a coloop of $(D|_{C})_{min}$, we have
$$r((D|_{C})_{min})\leq |C|-2.$$
Since $|C|$ and $r((D|_{C})_{min})$ have different parity, we see that  $$r((D|_{C})_{min})< |C|-2.$$ Any minor of a binary delta-matroid is binary, this gives $D|_{C}$ is a binary delta-matroid.

We claim that there exists a feasible set $F'\in \mathcal{F}(D|_{C})$ such that $|F'|=r((D|_{C})_{min})+1$. Otherwise, for any base $B$ of $(D|_{C})_{min}$, $(D|_{C})*B$ is a normal delta-matroid and $(D|_{C})*B$ doesn't contain a singleton. There is a binary antisymmetric matrix $A$ such that $D(A)=(D|_{C})*B$. It follows that $(D|_{C})*B$ is an even delta-matroid, a contradiction, since $D|_{C}$ is an odd delta-matroid and evenness is compatible with twist of delta-matroid. We see that $F'=C$ as above. But $$|F'|=r((D|_{C})_{min})+1<|C|-1,$$ a contradiction. Hence, $C$ is a circuit of $D_{min}$.

($\Leftarrow$)  Since $C$ is a circuit of $D_{min}$, we see that $C$ is a circuit of $(D|_{C})_{min}$. Then $C-e$ is a feasible set of $D|_{C}$ for any $e\in C$. Hence $D|_{C}$ is an odd delta-matroid by $C\in \mathcal{F}(D|_{C}).$ It follows that $D$ is an odd delta-matroid.
\end{proof}

\begin{remark}
The above does not hold for non-binary delta-matroid. For example, $$D=(\{1, 2, 3\}, \{\emptyset, \{1, 2\}, \{2, 3\}, \{1, 3\}, \{1, 2, 3\}\}).$$ $D$ is one of the minimal non-binary delta-matroids as shown in \cite{AB3} and  $$\mathcal{C}(D_{min})=\{\{1\}, \{2\}, \{3\}\}$$ but $$\mathcal{F}(D|_1)=\mathcal{F}(D|_2)=\mathcal{F}(D|_3)=\{\emptyset\},$$
that is, $$\{1\}\notin \mathcal{F}(D|_1), \{2\}\notin \mathcal{F}(D|_2), \{3\}\notin \mathcal{F}(D|_3).$$
\end{remark}

\begin{theorem}
Let $D=(E, \mathcal{F})$ be a binary even delta-matroid. Then $D$ is bipartite if and only if $D+E$ is an even delta-matroid.
\end{theorem}

\begin{proof}
($\Rightarrow$) Suppose that $D+E$ is an odd delta-matroid. Then there exists a circuit $C$ of $(D+E)_{min}$ such that $C\in \mathcal{F}((D+E)|_{C})$ by Lemma \ref{lec}. The bases of $(D+E)_{min}$ are the same as those of $D_{min}$, so $C$ is a circuit of $D_{min}$. Since $D$ is bipartite, the size of $C$ is even. It is easy to check that  $$(D+E)|_{C}=(D|_{C})+C$$ by the definition of loop complementation. Since $C$ is a circuit of $(D+E)_{min}$,  it follows that $C$ is a circuit of $$(D+E)_{min}|_{C}=((D+E)|_{C})_{min}=((D|_{C})+C)_{min}$$ and $$C\in \mathcal{F}((D+E)|_{C})=\mathcal{F}((D|_{C})+C).$$ Therefore $$\mathcal{F}((D|_{C})+C)=\{C-e: e\in C\} \cup \{C\}.$$ We know $X\in \mathcal{F}(D|_{C})$ if and only if $|\{Z \in \mathcal{F}((D|_{C})+C): Z \subseteq X\}|$ is odd. Hence $$\mathcal{F}(D|_{C})=\{C-e: e\in C\} \cup \{C\}.$$ Thus $D|_{C}$ is an odd delta-matroid. This contradicts the fact that $D$ is an even delta-matroid. Therefore $D+E$ is an even delta-matroid.

($\Leftarrow$) Assume that  $D$ is not bipartite. Then there exists a circuit $C$ of $D_{min}$ and the size of $C$ is odd.
It is easy to check that $C$ is a circuit of $D_{min}|_{C}$, that is, $C$ is a circuit of $(D|_{C})_{min}$.
Since $D$ is an even delta-matroid, $D|_{C}$ is also an even delta-matroid. Then $$\mathcal{F}(D|_{C})=\{C-e: e\in C\}.$$ It follows that $$\mathcal{F}((D|_{C})+C)=\{C-e: e\in C\}\cup \{C\}.$$Thus $(D|_{C})+C$ is an odd delta-matroid, that is, $(D+E)|_{C}$ is an odd delta-matroid. This contradicts the fact that $D+E$ is an even delta-matroid. Hence $D$ is bipartite.
\end{proof}

A standard result in graph theory is that a plane graph $G$ is Eulerian if and only if the dual of $G$ is bipartite. This result also holds for binary matroids. This fact only holds for general ribbon graphs in one direction: the dual of any bipartite ribbon graph is Eulerian.  In the following, we extend the result from ribbon graphs to the class of delta-matroids that arise as twists of binary matroids.

\begin{lemma}[\cite{DW}]\label{led}
A bianry matroid is Eulerian if and only if its dual matroid is bipartite.
\end{lemma}

\begin{lemma}[\cite{CISR}]\label{lee}
Let $M=(E, \mathcal{I})$ be a matroid and $A\subseteq E$. Let $D=M*A$. Then $D_{min}=M/A\oplus (M\backslash A^{c})^{*}$ and $D_{max}=M\backslash A\oplus (M/ A^{c})^{*}$.
\end{lemma}

\begin{lemma}[\cite{JO}]\label{lef}
Let $C$ be a circuit of a binary matroid $M$ and $e$ be an element of $E(M)-C$. Then, in $M/e$, either $C$ is a circuit, or $C$ is a disjoint union of two circuits.
\end{lemma}

\begin{theorem}\label{tha}
Let $M=(E, \mathcal{I})$ be a binary matroid, $A\subseteq E$ and $D=M*A$. If $D$ is bipartite, then $D^{*}$ is Eulerian.
\end{theorem}

\begin{proof}
Firstly, we show that a circuit $C$ of $M\setminus A$ is a disjoint union of some circuits of $M/A$. If $A=\emptyset$, then there is nothing to prove, so we can
assume that $A=\{e_{1}, \cdots, e_{n}\}$. Then either $C$ is a circuit of $M/e_{1}$, or $C=C_{1}\cup C_{2}$ which $C_{1}, C_{2}$  are two disjoint circuits of $M/e_{1}$ by Lemma \ref{lef}. Repeat the process above, it follows that $C$ is a disjoint union of some circuits of $M/A$.  For any circuit $C'$ of $D_{max}$, since $D_{max}=M\backslash A\oplus (M/ A^{c})^{*}=M\backslash A\oplus M^{*}\backslash A^{c}$ by Lemma \ref{lee}, $C'$ is a circuit of $M\backslash A$ or $M^{*}\backslash A^{c}$. Thus $C'$ is a disjoint union of some circuits of $M/A$ or $M^{*}/A^{c}$. Since $D_{min}=M/A\oplus (M\backslash A^{c})^{*}=M/A\oplus M^{*}/A^{c}$ by Lemma \ref{lee} and $D$ is bipartite, it follows that all of the sizes of circuits of $M/A$ and $M^{*}/A^{c}$ are even. Therefore the size of $C'$ is even. Hence $D_{max}$ is bipartite.  It is easy to check that $D_{max}=((D^{*})_{min})^{*}$. By Lemma \ref{led}, $(D^{*})_{min}$ is Eulerian, that is, $D^{*}$ is Eulerian.
\end{proof}

\begin{remark}
The converse of Theorem \ref{tha} is not true. A counterexample is given as shown below. Let $M=(\{1, 2\}, \{\{1\}, \{2\}\})$ and $D=M*1=(\{1, 2\}, \{\emptyset, \{1, 2\}\})$. Then $D^{*}=(\{1, 2\}, \{\emptyset, \{1, 2\}\})$. It is easy to check that $D^{*}$ is Eulerian, but $D$ is not bipartite.
\end{remark}

\begin{theorem}\label{thc}
Let $M=(E, \mathcal{I})$ be a binary matroid, $A\subseteq E$ and $D=M*A$. Then:

\begin{enumerate}
\item $D$ is bipartite if and only if $M\setminus A^{c}$ and $M^{*}\setminus A$ are Eulerian;
\item $D$ is Eulerian if and only if $M\setminus A^{c}$ and $M^{*}\setminus A$ are bipartite.
\end{enumerate}

\end{theorem}

\begin{proof}
By Lemma \ref{lee}, $D$ is bipartite (Eulerian) if and only if both $M/A$ and $(M\backslash A^{c})^{*}$ are bipartite (Eulerian) if and only if both $(M/A)^{*}$ and $M\backslash A^{c}$ are Eulerian (bipartite) by reason that any minor of a binary matroid is binary and Lemma \ref{led}. So if and only if $M\setminus A^{c}$ and $M^{*}\setminus A$ are Eulerian (bipartite).
\end{proof}

\begin{remark}
Huggett and Moffatt (see \cite{SHM} Theorem 1.2) obtained the same results as Theorem \ref{thc} for embedded graphs.
\end{remark}

\begin{lemma}\label{leg}
If $D=(E, \mathcal{F})$ is a bipartite delta-matroid and $A\subseteq E$, then $D\setminus A$ is a bipartite delta-matroid.
\end{lemma}

\begin{proof}
The result is easily verified when $A =\emptyset$.
We just need to verify that when $A=\{e\}$. Since $(D\setminus e)_{min}=D_{min}\setminus e$ by Lemma \ref{leb}, it follows that $$\mathcal{C}((D\setminus e)_{min}) = \mathcal{C}(D_{min}\setminus e)=\{C\subseteq E-e: C\in \mathcal{C}(D_{min})\}.$$ Thus, $D\setminus e$ is a bipartite delta-matroid.
\end{proof}

\begin{theorem}
Let $D=(E, \mathcal{F})$ be a delta-matroid and $A\subseteq E$. If $D*A$ is bipartite, then $D^{*}/A^{c}$ and $D/A$
are both bipartite.
\end{theorem}

\begin{proof}
$D^{*}/A^{c}=((D^{*})*A^{c})\setminus A^{c}=(D*{A})\setminus A^{c}$, where the first by the relation between twist and contraction, and the second by the basic properties of twist. Similarly, $D/A=(D*A)\setminus A.$ Since $D*A$ is bipartite, then $(D*{A})\setminus A^{c}$ and $(D*A)\setminus A$ are bipartite by Lemma \ref{leg}, completing the proof.
\end{proof}

\section*{Acknowledgements}

We thank Graham Farr and Benjamin Jones for many helpful suggestions and comments. This work is supported by NSFC (No. 11671336) and the Fundamental
Research Funds for the Central Universities (No. 20720190062).


\begin{thebibliography}{9}
\bibitem{AB1} A. Bouchet, Greedy algorithm and symmetric matroids,
\emph{Math. Program.} \textbf{38} (1987) 147--159.

\bibitem{AB4}  A. Bouchet, Representability of $\Delta$-matroids, \emph{Colloq. Math. Soc. J$\acute{a}$nos Bolyai} (1987) 167--182.

\bibitem{AB2}  A. Bouchet, Maps and delta-matroids, \emph{Discrete Math.} \textbf{78} (1989) 59--71.

\bibitem{AB3}  A. Bouchet and A. Duchamp, Representability of delta-matroids over $GF(2)$, \emph{Linear Algebra Appl.} \textbf{146} (1991) 67--78.

\bibitem{RH}  R. Brijder and H. Hoogeboom, The group structure of pivot and loop complementation on graphs and set systems, \emph{European J. Combin.} \textbf{32} (2011) 1353--1367.

\bibitem{RH2}  R. Brijder and H. Hoogeboom, Quaternary bicycle matroids and the Penrose polynomial for delta-matroids, Preprint, arXiv: 1210.7718.

\bibitem{CISR} C. Chun, I. Moffatt, S. D. Noble and R. Rueckriemen, On the interplay between embedded graphs and delta-matroids, \emph{Proc. London Math. Soc.} \textbf{118} (2019) 3: 675--700.

\bibitem{CMNR} C. Chun, I. Moffatt, S. D. Noble and R. Rueckriemen, Matroids, delta-matroids and embedded graphs, \emph{J. Combin. Theory A.} \textbf{167} (2019) 7--59.

\bibitem{EM1}  J. A. Ellis-Monaghan and I. Moffatt, Twisted duality for embedded graphs, \emph{Trans. Amer. Math. Soc.} \textbf{364} (2012) 1529--1569.

\bibitem{EM}  J. A. Ellis-Monaghan and I. Moffatt, Graphs on surfaces, Springer New York, 2013.

\bibitem{SHM} S. Huggett and I. Moffatt, Bipartite partial duals and circuits in medial graphs, \emph{Combinatorica} \textbf{33} (2013) 231--252.

\bibitem{JO} J. Oxley, Matroid theory, 2nd edn, Oxford University Press, New York, 2011.

\bibitem{SR} M. M. Shikare and T. T. Raghunathan, A characterization of binary Eulerian matroids, \emph{Indian J. Pure Appl. Math.}  \textbf{27} (1996) 2: 153--155.

\bibitem{MMS} M. M. Shikare, New characterizations of Eulerian and bipartite binary matroids, \emph{Indian J. Pure Appl. Math.}  \textbf{32} (2001) 2: 215--219.

\bibitem{DW} D. Welsh, Euler and bipartite matroids, \emph{J. Combin. Theory}  \textbf{6} (1969) 375--377.

\bibitem{PJW} P. J. Wilde, The Euler circuit theorem for binary matroids, \emph{J. Combin. Theory B.}  \textbf{18} (1975) 260--264.

\bibitem{AZI} A. \v{Z}itnik, Plane graphs with Eulerian Petrie walks, \emph{Discrete Math.} \textbf{244} (2002) 539--549.

\end{thebibliography}
\end{document}